\documentclass[reqno]{amsart}
\usepackage[]{latexsym}
\usepackage[]{amssymb}
\usepackage[]{amsmath}
\usepackage{amsfonts}
\usepackage{amsthm}
\usepackage{graphicx}
\usepackage[all]{xy}
\usepackage[]{mathrsfs}
\usepackage{enumerate}
\usepackage{color}
\usepackage[capitalise]{cleveref}

\newtheorem{theorem}{Theorem}[section]

\newtheorem{thm}[theorem]{Theorem}
\newtheorem{cor}[theorem]{Corollary}
\newtheorem{prop}[theorem]{Proposition}

\theoremstyle{definition}

\numberwithin{equation}{theorem}

\DeclareMathAlphabet{\mathpzc}{OT1}{pzc}{m}{it}

\DeclareMathOperator{\kar}{char}
\makeatletter
\newcommand{\bigperp}{%
  \mathop{\mathpalette\bigp@rp\relax}%
  \displaylimits
}

\newcommand{\bigp@rp}[2]{%
  \vcenter{
    \m@th\hbox{\scalebox{\ifx#1\displaystyle2.1\else1.5\fi}{$#1\perp$}}
  }%
}
\makeatother
\begin{document}

\newcommand{\str}{{\mathtt{str}}}
\newcommand{\Trd}{{\mathrm{Trd}}}
\newcommand{\rad}{{\mathrm{rad}}}
\newcommand{\id}{{\mathrm{id}}}
\newcommand{\Ad}{{\mathrm{Ad}}}
\newcommand{\Ker}{{\mathrm{Ker}}}
\newcommand{\wedges}[1]{\d b_1\wedge\ldots\wedge \d b_{#1}}
\newcommand{\rank}{{\mathrm{rank}}}
\renewcommand{\dim}{{\mathrm{dim}}}
\newcommand{\coker}{{\mathrm{Coker}}}
\newcommand{\can}{\overline{\rule{2.5mm}{0mm}\rule{0mm}{4pt}}}
\newcommand{\End}{{\mathrm{End}}}
\newcommand{\Sand}{{\mathrm{Sand}}}
\newcommand{\Hom}{{\mathrm{Hom}}}
\newcommand{\Nrd}{{\mathrm{Nrd}}}
\newcommand{\Srd}{{\mathrm{Srd}}}
\newcommand{\ad}{{\mathrm{ad}}}
\newcommand{\rk}{{\mathrm{rk}}}
\newcommand{\Mon}{{\mathrm{Mon}}}
\newcommand{\disc}{{\mathrm{disc}}}
\newcommand{\Sym}{{\mathrm{Sym}}}
\newcommand{\Skew}{{\mathrm{Skew}}}
\newcommand{\Nrp}{{\mathrm{Nrp}}}
\newcommand{\Trp}{{\mathrm{Trp}}}
\newcommand{\Symd}{{\mathrm{Symd}}}
\renewcommand{\dir}{{\mathrm{dir}}}
\renewcommand{\geq}{\geqslant}
\renewcommand{\leq}{\leqslant}
\newcommand{\an}{{\mathrm{an}}}
\renewcommand{\Im}{{\mathrm{Im}}}
\newcommand{\Int}{{\mathrm{Int}}}
\renewcommand{\d}{{\mathrm{d}}}
\newcommand{\qp}[2]{\mbox{$#1 \otimes^\mathrm{qp}#2 $}}
\newcommand{\qf}[1]{\mbox{$\langle #1\rangle $}}
\newcommand{\pff}[1]{\mbox{$\langle\!\langle #1
\rangle\!\rangle $}}
\newcommand{\pfr}[1]{\mbox{$\langle\!\langle #1 ]]$}}
\newcommand{\HH}{{\mathbb H}}
\newcommand{\s}{{\sigma}}
\newcommand{\lra}{{\longrightarrow}}
\newcommand{\ZZ}{{\mathbb Z}}
\newcommand{\NN}{{\mathbb N}}
\newcommand{\FF}{{\mathbb F}}
\newcommand{\PERP}{\mbox{\raisebox{-.5ex}{{\Huge $\perp$}}}}
\newcommand{\Perp}{\mbox{\raisebox{-.2ex}{{\Large $\perp$}}}}
\newcommand{\M}[1]{\mathbb{M}( #1)}

\newcommand{\vf}{\varphi}
\newcommand{\mg}[1]{#1^{\times}}

\title{Totally decomposable symplectic and unitary involutions}
\email{Andrew.Dolphin@uantwerpen.be}
\address{Universiteit Antwerpen, Departement Wiskunde-Informatica, Middelheim\-laan~1, 2020 Antwerpen, Belgium.}
\author{Andrew Dolphin}
\thanks{This work was supported by the {Deutsche Forschungsgemeinschaft} (project \emph{The Pfister Factor Conjecture in Characteristic Two}, BE 2614/4) and  the FWO Odysseus programme (project \emph{Explicit Methods in Quadratic Form Theory})}

\begin{abstract} We study totally decomposable symplectic and unitary  involutions on  central simple algebras of index $2$ and on  split central simple algebras respectively.  We show that for every field extension, these involutions are either anisotropic or hyperbolic after extending scalars, and that the converse holds if the algebras are of $2$-power  degree. These results are new in characteristic $2$, otherwise were shown in  \cite{Becher:qfconj} and \cite{Black:inv} respectively.


\medskip\noindent
\emph{Mathematics Subject Classification (MSC 2010):} 16W10, 16K20, 11E39, 11E81, 12F05. 
\end{abstract}

\maketitle

\section{Introduction}

%

An algebra with involution is totally decomposable if it is isomorphic to a tensor product of quaternion algebras with involution. Over fields of characteristic different from two, the adjoint involution of a Pfister form is a totally decomposable involution.
Pfister forms are a central part of the algebraic theory of quadratic forms, and in  \cite{parimala:pfisterinv} it was asked whether totally decomposable involutions share certain characterising properties of Pfister forms. In particular, whether totally decomposable involutions are exactly those involutions on algebras of two-power degree that are either anisotropic or hyperbolic after extending scalars. 
That  totally decomposable orthogonal involutions over a field of characteristic different from two are always either anisotropic or hyperbolic can be shown (see  \cite[(3.2)]{Black:inv}) using the 
 non-hyperbolic splitting results of Karpenko, \cite{karpenko:hyporth}, and that any totally decomposable orthogonal  involution on a split algebra is adjoint to a Pfister form. The later result is known  as the Pfister Factor Conjecture, and was shown in  \cite{Becher:qfconj}. 
 The converse, that any orthogonal involution on an algebra of two-power degree that is anisotropic or hyperbolic after extending scalars to any field extension is necessarily  totally decomposable, is clear for split algebras, but  otherwise  remains largely open.
In   \cite{dolphin:orthpfist}, this question is considered for orthogonal involutions in characteristic two, whose behaviour is somewhat unusual.

Here we consider analogues  of the property that a   totally decomposable  orthogonal involution  on a split algebra is  either anisotropic or hyperbolic, and its converse, for symplectic and unitary involutions.
 As every split symplectic involution is hyperbolic, and hence the question is trivial in this case, we consider symplectic involutions on an algebra Brauer equivalent to a quaternion algebra. 
 We show that for every field extension, totally decomposable symplectic involutions on index two algebras and unitary involutions on split algebras  are either anisotropic or hyperbolic after extending scalars, and that the converse holds if the algebras are of two-power  degree.
These results were first proven in  \cite{Becher:qfconj} and \cite[(3.1)]{Black:inv} respectively under the assumption that the base field was of characteristic different from two. Here we make no assumption on the characteristic of the base field.

Our approach follows that used in   \cite{Becher:qfconj}. The  main tool is  an index reduction argument using the 
 the so-called trace forms of Jacobson, first introduced for characteristic different from two in  \cite{jacobson:hermformstr}. These trace forms give a  correspondence between hermitian forms over a quaternion algebra  or a quadratic separable extension with the respective canonical involutions  and certain quadratic forms over the base field. In sections \S\ref{section:jacobson} and \S\ref{jacanal} we consider versions of these results that hold in any characteristic and derive analogous statements for involutions.


\section{Algebras with involution}\label{section:basicsalgs}

Throughout, let $F$ be a field,  $\kar(F)$ denote its characteristic and  $F^\times$ denote its multiplicative group. 
We refer to \cite{pierce:1982} as a general reference  on finite-dimensional algebras over fields, and for central simple algebras in particular, and to \cite{Knus:1998} for involutions.
Let $A$ be an (associative) $F$-algebra.  We denote the centre of $A$ by $Z(A)$.
For a field extension $K/F$, the $K$-algebra $A\otimes_F K$ is denoted by $A_K$.
 An $F$-\emph{involution} on $A$ is an $F$-linear map $\sigma:A\rightarrow A$ such that  $\sigma(xy)=\sigma(y)\sigma(x)$ for all $x,y\in A$ and  $\sigma^2=\textrm{id}_A$.


Assume now that $A$ is finite-dimensional and simple  (i.e.~it has no nontrivial two sided ideals).
Then $Z(A)$ is a field, and  by Wedderburn's Theorem (see \cite[(1.1)]{Knus:1998}) we have  $A\simeq\End_D(V)$ for an  $F$-division algebra $D$  and a right $D$-vector space $V$, and furthermore $\dim_{Z(A)}(A)$ is a square number, whose positive square root is called the \emph{degree of $A$} and is denoted $\mathrm{deg}(A)$. The degree of $D$ is called the \emph{index of $A$} and is denoted $\mathrm{ind}(A)$. We call $A$ \emph{split} if $\mathrm{ind}(A)=1$. 
We call a field extension $K/F$ a \emph{splitting field of $A$} if $A_K$ is split.  If $Z(A)=F$, then we call the $F$-algebra $A$ \emph{central simple}.  Two central simple $F$-algebras $A$ and $B$ are called \emph{Brauer equivalent} if $A$ and $B$ are isomorphic to endomorphism algebras of two right vector spaces over the same $F$-division algebra. 
If $A$ is a central simple $F$-algebra let $\Trd_A:A\lra F$ denote the reduced trace map  and $\Nrd_A:A\lra F$ the reduced norm map (see \cite[(1.6)]{Knus:1998} for the definitions).

 An \emph{$F$-algebra with involution} is a pair $(A,\sigma)$ of a finite-dimensional $F$-algebra $A$ and an $F$-involution $\sigma$ on $A$ such that one has $F=\{x \in Z(A) \mid \sigma(x)=x\}$, and such that either $A$ is simple or $A$ is a product of two simple $F$-algebras that are mapped to one  another by $\sigma$. In this situation there are two possibilities: either $Z(A)=F$, so that $A$ is a central simple $F$-algebra, or $Z(A)/F$ is a quadratic \'etale extension with $\sigma$ restricting to the nontrivial $F$-automorphism of $Z(A)$. To distinguish these two situations, we speak of  algebras with involution of the \emph{first} and \emph{second kind}:  we say that the $F$-algebra with involution $(A,\sigma)$ is of the \emph{first kind} if $Z(A)=F$ and of the \emph{second kind} otherwise. For more information on involutions of the second kind, also called \emph{unitary involutions}, we refer to  \cite[\S 2.B]{Knus:1998}.  Involutions of the first kind are divided into two \emph{types}, \emph{orthogonal} and \emph{symplectic} (see \S\ref{section:her}).

Let $(A,\s)$ be an $F$-algebra with involution.
If $Z(A)$ is a field, then $A$ is a central simple $Z(A)$-algebra, and we say  $(A,\s)$ is \emph{split} if $A$ is split as $Z(A)$-algebra.
If $Z(A)\simeq F\times F$, then $(A,\s)\simeq (B\times B^{\mathsf{op}},\epsilon)$ where $B$ is a central simple $F$-algebra, $B^{\mathsf{op}}$ is its opposite algebra and $\epsilon$ is the map exchanging the components of elements of $B\times B^{\mathsf{op}}$; in this case we say  $(A,\s)$ is \emph{split} if $B$ is split as an $F$-algebra.
Given a field extension $K/F$, we abbreviate  $\sigma_K=\sigma\otimes \textrm{id}_K$ and $(A,\sigma)_K=(A_K,\sigma_K)$ is a $K$-algebra with involution. We call $K$ a \emph{splitting field} of $(A,\sigma)$ if $(A,\s)_K$ is split.

Let  $(A,\s)$ and $(B,\tau)$ be $F$-algebras with involution.
Letting $(\sigma\otimes\tau)(a\otimes b)=\s(a)\otimes \tau(b)$ for $a\in A$ and $b\in B$ determines an $F$-involution  $\s\otimes \tau$ on the $F$-algebra $A\otimes_F B$. We denote the pair $(A\otimes_F B,\sigma\otimes \tau)$ by $(A,\sigma)\otimes(B,\tau)$.
By an \emph{isomorphism of $F$-algebras with involution}   we mean an $F$-algebra isomorphism $\Phi: A\rightarrow B$ satisfying $\Phi\circ\sigma=\tau\circ\Phi$. 
Let $(A,\s)$ be an $F$-algebra with involution with centre $K$.
We call  $(A,\s)$ \emph{totally decomposable} if there exists an $n\in\mathbb{N}$ and  $F$-quaternion algebras with involution $(Q_1,\s_1),\ldots, (Q_n,\s_n)$ with common centre $K$ such that $(A,\sigma)\simeq \bigotimes_{i=1}^n (Q_i,\sigma_i)$.

 We call  $(A,\sigma)$ \emph{isotropic} if there exists an element $a\in A\setminus\{0\}$ such that $\sigma(a)a=0$, and \emph{anisotropic} otherwise.
 We call an idempotent $e$ of $A$ \emph{hyperbolic with respect to $\sigma$} if  $\sigma(e)=1-e$. We call an $F$-algebra with involution $(A,\sigma)$ \emph{hyperbolic}  if $A$ contains a hyperbolic idempotent with respect to $\sigma$. By  \cite[(12.35)]{Knus:1998} we have:

\begin{prop}\label{cor:allhypsame}
Let $(A,\s_1)$ and $(A,\s_2)$ be hyperbolic $F$-algebras with involution such that $\s_1|_{Z(A)}=\s_2|_{Z(A)}$. Then $(A,\s_1)\simeq(A,\s_2)$.
\end{prop}

Let $(A,\sigma)$ be an $F$-algebra with involution. For $\lambda \in Z(A)$ such that $\lambda\sigma(\lambda)=1$, let $\mathrm{Sym}_{\lambda}(A,\sigma)= \{ a\in A\mid \lambda\sigma(a)=a\}$  and  $\Symd_\lambda(A,\sigma)= \{ a+\lambda\sigma( a)\,|\, a\in A\}$. These are $F$-linear subspaces of $A$ and 
we write $\mathrm{Sym}(A,\sigma)= \mathrm{Sym}_1(A,\sigma)$ and $\Symd(A,\sigma)= \Symd_{1}(A,\sigma)$.

We also need to consider quadratic pairs for our main results. 
 Let $(A,\s)$ be an $F$-algebra with involution of the first kind.
We call an $F$-linear map $f:\Sym(A,\sigma)\rightarrow F$ a \emph{semi-trace on $(A,\sigma)$} if it satisfies $f(x+\sigma(x)) = \mathrm{Trd}_A(x)$  for all $x\in A$. 
If  $\kar(F)=2$, then the existence of a semi-trace on $(A,\sigma)$ implies that $\Trd_A(\Sym(A,\sigma))=\{0\}$ and hence by \cite[(2.6)]{Knus:1998} that $(A,\sigma)$  is symplectic.  
%
%
An \emph{$F$-algebra with quadratic pair} is a triple $(A,\sigma,f)$ where $(A,\sigma)$ is an $F$-algebra with involution of the first kind, assumed to be orthogonal if $\kar(F)\neq 2$ and symplectic if $\kar(F)=2$, and where $f$ is a semi-trace on $(A,\sigma)$.
If $\kar(F)\neq2$, then the semi-trace $f$ is uniquely determined  by $(A,\sigma)$ (see \cite[p.56]{Knus:1998}).
Hence in characteristic different from $2$ the concept of an algebra with quadratic pair is  equivalent to the concept of an algebra with orthogonal involution. 
An isomorphism of quadratic pairs $(A,\sigma,f)$ and $(B,\tau,g)$ is an isomorphism $\Phi$ of the underlying $F$-algebras with involution satisfying $ f=g\circ\Phi$.

Let $(A,\s)$ and $(B,\tau)$ be $F$-algebras with involution of the first kind and let $f$ be a semi-trace on $(A,\s)$.
There is a unique semi-trace $g$ on $(B,\tau)\otimes (A,\s)$ such that $g(b\otimes a) =  \Trd_B(b)\cdot f(a)$
for all $a\in\Sym(A,\s)$ and $b\in \Sym(B,\tau)$ (see  \cite[(5.18)]{Knus:1998} for the case of $\kar(F)=2$, otherwise this is trivial). 
Assume that $(A,\s)$ and $(B,\tau)$  are orthogonal if $\kar(F)\neq2$. Then by \cite[(2.23)]{Knus:1998}, $(B,\tau)\otimes (A,\s)$  is orthogonal if $\kar(F)\neq 2$ and symplectic if $\kar(F)=2$.
Hence   we obtain an $F$-algebra with quadratic pair $( B\otimes_F A, \tau\otimes \s,g)$, which we denote by $(B,\tau)\otimes(A,\sigma, f)$.
  By \cite[(5.3)]{dolphin:PFC}, the tensor product of algebras with involution and the tensor product of an algebra with involution and an algebra with quadratic pair are mutually associative operations. In particular, for an $F$-algebra with quadratic pair $(A,\s,f)$ and $F$-algebras with involution of the first kind $(B,\tau)$ and $(C,\gamma)$, we may write $(C,\gamma)\otimes(B,\tau)\otimes(A,\s,f)$ without any ambiguity.
  
Assume now that $(A,\s)$ and $(B,\tau)$ are symplectic. 
We may  define a semi-trace $h$ on $(B,\tau)\otimes (A,\s)$ in the following manner.
If $\kar(F)=2$, then by \cite[(5.20)]{Knus:1998} there exists a unique semi-trace $h$ on 
 $(B,\tau)\otimes(A,\s)$ such that $h(s_1\otimes s_2)=0$ for all $s_1\in \Sym(A,\s)$ and $s_2\in \Sym(B,\tau)$. 
  If $\kar(F)\neq2$, then  $(B,\tau)\otimes(A,\s)$ is orthogonal by \cite[(2.23)]{Knus:1998} and we 
  let $h=\frac{1}{2}\Trd_{A\otimes_F B}$. 
  In either case, we denote the $F$-algebra with quadratic pair $(B\otimes_FA,\tau\otimes\s,h)$ by $(B,\tau)\boxtimes (A,\s)$.
Note that by \cite[(5.4)]{dolphin:PFC},  $(B,\tau)\boxtimes (A,\s)\simeq  (B,\tau)\otimes (A,\s,f)$ for any choice of semi-trace $f$ on $(A,\s)$.
Moreover, for any algebra with involution of the first kind $(C,\gamma)$ we have 
$$\bigl((C,\gamma)\otimes (B,\tau)\bigr)\boxtimes (A,\sigma)\simeq (C,\gamma)\otimes(B,\tau)\otimes(A,\sigma,f)\simeq (C,\gamma)\otimes\bigl((B,\tau)\boxtimes(A,\sigma)\bigr)\,.$$ We therefore  use the notation $(C,\gamma)\otimes(B,\tau)\boxtimes (A,\sigma)$ for this tensor product.


\section{Hermitian and quadratic  forms}\label{section:her}

We  recall  certain  results we use from hermitian and quadratic  form theory. 
We refer to \cite[Chapter 1]{{Knus:1991}} and \cite[Chapters 1 and 2]{Elman:2008} for standard notation and terminology, and  as  general references on hermitian and quadratic forms respectively.

Throughout this section, let  $(D,\theta)$ be an $F$-division algebra with involution.
Fix $\lambda\in Z(D)$ such that $\lambda\theta(\lambda)=1$.  Note that if $(D,\theta)$ is of the first kind one must have that $\lambda=\pm1$. 
A \emph{$\lambda$-hermitian form over $(D,\theta)$} is a pair $(V,h)$ where $V$ is a finite-dimensional right $D$-vector space and $h$ is
 a bi-additive map  $h:V\times V\rightarrow D$ such that $h(xd_1,yd_2)=\theta(d_1)h(x,y)d_2$ 
 and $h(y,x)=\lambda\theta({h(x,y)})$  hold 
 for all  $x,y\in V$ and  $d_1,d_2\in D$.
If $\lambda=1$ then we call a $\lambda$-hermitian form simply an \emph{hermitian form}.
 If $h(x,y)=0$ for all $y\in V$ implies $x=0$, then we say that $(V,h)$ is 
\emph{nondegenerate}. 
We say $(V,h)$ \emph{represents an element $a\in D$} if $h(x,x)=a$ for some $x\in V\setminus\{0\}$. We call $(V,h)$ \emph{isotropic} if it represents $0$, and \emph{anisotropic} otherwise.

We say that a $\lambda$-hermitian form $(V,h)$ over $(D,\theta)$ is \emph{even} if $h(x,x)\in \Symd_{\lambda}(D,\theta)$ for all $x\in V$. Note that if $\kar(F)\neq 2$ or $(D,\theta)$ is unitary then all $\lambda$-hermitian forms over $(D,\theta)$ are even as in these cases $\Symd_{\lambda}(D,\theta)= \mathrm{Sym}_\lambda(D,\theta)$ (see \cite[(2.A) and (2.17)]{Knus:1998}). 
Assume $(V,h)$ is even. 
We say that $(V,h)$ is \emph{hyperbolic} if there exists a totally isotropic subspace $W\subseteq V$ with $\dim_D(W)=\frac{1}{2}\dim(V,h)$. The following proposition follows easily from \cite[Chapter 1, (3.7.3)]{Knus:1991}.

\begin{prop}\label{prop:hypuniq} Let $\varphi$ and $\psi$ be hyperbolic even $\lambda$-hermitian forms over $(D,\theta)$. If $\dim(\varphi)=\dim(\psi)$ then $\varphi\simeq \psi$.
\end{prop}

%


\begin{prop}\label{cor:hypiso}
Let $\varphi$ and $\psi$ be even $\lambda$-hermitian forms over $(D,\theta)$ of the same dimension. Then $\varphi\simeq \psi$ if and only if $\varphi\perp(-\psi)$ is hyperbolic.
\end{prop}
\begin{proof}
That $\varphi\perp(-\psi)$ is hyperbolic if $\varphi\simeq \psi$  is clear. For the converse, we have that  $\varphi\perp(-\psi)$ and $\psi\perp (-\psi)$  are hyperbolic and of the same dimension. Therefore by \cref{prop:hypuniq}  we have $\varphi\perp(-\psi)\simeq \psi\perp (-\psi)$.
As $-\psi$ is even we have  $\varphi\simeq \psi$ by \cite[Chapter 1, (6.4.5)]{{Knus:1991}}.
\end{proof}

 For every nondegenerate $\lambda$-hermitian form $\varphi=(V,h)$ over $(D,\theta)$, there is a unique $F$-involution on $\End_D(V)$, denoted $\ad_h$,  with  $\ad_h(a)=\theta(a)$ for all $a\in E$ such that
$h(f(x),y)=h(x,\ad_h(f)(y)) \textrm{ for all } x,y\in V \textrm{ and } f\in \End_D(V)$ (see \cite[(4.1)]{Knus:1998}).
We denote the $F$-algebra with involution $(\End_D(V),\ad_{h})$ by $\Ad(\varphi)$ and we refer to it as the \emph{$F$-algebra with involution adjoint to $\varphi$}.
We denote the  $1$-dimensional hermitian form $(D,h)$ over $(D,\theta)$ given by $h(x,y)=\theta(x)y$ for $x,y\in D$ by $\qf{1}_{(D,\theta)}$.  It is easy to see that  $\Ad(\qf{1}_{(D,\theta)})\simeq (D,\theta)$.



By a \emph{bilinear form over $F$} we mean a $\lambda$-hermitian form over $(F,\id)$ (in this case we must have $\lambda=\pm1$).
Let $\varphi=(V,b)$ be a bilinear form over $F$. 
We call $\vf$  \emph{symmetric} if $b(x,y)=b(y,x)$ for all $x,y\in V$, i.e.~$\varphi$ is a hermitian form over $(F,\id)$. 
 We call $\varphi$ \emph{alternating} if  $b(x,x)=0$ for all $x\in V$. This is equivalent to $\varphi$ being an even $(-1)$-hermitian form over $(F,\id)$, as $\Symd_{-1}(F,\id)=\{0\}$.  
Any split $F$-algebra with involution of the first kind is isomorphic to $\Ad(\varphi)$ for some nondegenerate symmetric or alternating bilinear form $\varphi$ over $F$ (see \cite[(2.1)]{Knus:1998}).
An $F$-algebra with involution of the first kind is said to be \emph{symplectic} if it is adjoint to an alternating bilinear form over some splitting field, and \emph{orthogonal} otherwise.  



Let $\varphi=(V,b)$ be a symmetric bilinear form over $F$ and  $\psi=(W,h)$ be a $\lambda$-hermitian form over   $(D,\theta)$. Then $V\otimes_FW$ is a finite dimensional right $D$-vector space. Further, 
 there is a unique $F$-bilinear map $b\otimes h:(V\otimes_F W)\times (V\otimes_F W)\rightarrow F$ such that 
$(b\otimes h)\left( (v_1\otimes w_1), (v_2\otimes w_2)\right) =b(v_1,v_2)\cdot h(w_1,w_2) $
for all $w_1,w_2\in W, v_1,v_2\in V$. We call the $\lambda$-hermitian form $(V\otimes_F W, b\otimes h)$ over $(D,\theta)$ the  \emph{tensor product of $\varphi$ and $\psi$}, and denote it by $\varphi\otimes \psi$ (cf.~ \cite[Chapter 1, \S8]{Knus:1991}).

For $b_1,\ldots,b_n\in F^\times$ the symmetric $F$-bilinear map $b:F^n\times F^n\rightarrow F$  given by $(x,y)\mapsto \sum_{i=1}^n b_ix_iy_i$ yields a
symmetric bilinear form $(F^n,b)$,  which we denote by $\qf{b_1,\ldots,b_n}$.
For a nonnegative integer $m$, by an \emph{$m$-fold bilinear Pfister form} we mean a nondegenerate symmetric bilinear form that is isometric to $\qf{1,a_1}\otimes\ldots\otimes\qf{1,a_m}$  for some  $a_1,\ldots,a_m\in F^\times$, and we denote this form by $\pff{a_1,\ldots, a_m}$. We call $\qf{1}$ the 
 the \emph{$0$-fold bilinear Pfister form}.






By a \textit{quadratic form over $F$} we  mean a pair $(V,q)$ of a finite-dimensional $F$-vector space $V$ and a map  $q:V\rightarrow F$ such that, firstly,  $q(\lambda x)=\lambda^2q(x)$ holds for all $x\in V$ and $\lambda\in F$, and secondly,  the map   $b_q:V\times V\rightarrow F\,,\,(x,y)\longmapsto q(x+y)-q(x)-q(y)$ is  $F$-bilinear. 
 We say that $(V,q)$ is \emph{nonsingular} if $b_q$ is nondegenerate. Quadratic forms correspond on a one-to-one basis to symmetric bilinear forms if $\kar(F)\neq2$, but not otherwise.
We say two quadratic forms $\rho_1$ and $\rho_2$ over $F$ are \emph{similar} if there exists an element  $c\in F^\times$ such that $c\rho_1\simeq\rho_2$.

Recall that there is a natural notion of the tensor product of a symmetric bilinear form $\varphi$ and a quadratic form $\rho$ over $F$, denoted $\varphi\otimes\rho$  (see \cite[p.51]{Elman:2008}).  
For a nonnegative integer $m$, by an \emph{$m$-fold quadratic Pfister form over $F$} (or simply an  \emph{$m$-fold Pfister form}) we mean a quadratic form that is isometric to the tensor product of a $2$-dimensional nonsingular quadratic form representing $1$ and an $(m-1)$-fold bilinear Pfister form over $F$.
 For the following result see  \cite[(9.10) and (23.4)]{Elman:2008}.

\begin{prop}\label{Pfister} Let $m$ be a nonnegative integer. 
A nonsingular  quadratic form $\rho$ with $\dim(\rho)=2^m$ is similar to a Pfister form if and only if for every field extension  $K/F$, $\pi_K$ is either anisotropic or hyperbolic.
\end{prop}

Quadratic pairs correspond to quadratic forms a way similar to the correspondence between  involutions of the first kind and bilinear forms.
To any nonsingular quadratic form  $\rho=(V,q)$ over $F$ one may associated a uniquely determined  quadratic pair on $\End_F(V)$, giving the  \emph{$F$-algebra with quadratic pair adjoint to $\rho$}, denoted $\Ad({\rho})$. See \cite[(5.11)]{Knus:1998}  for a description. By \cite[(5.11)]{Knus:1998}, any quadratic pair on a split algebra $\End_F(V)$ is the adjoint of some nonsingular quadratic form over $F$.
The notions of isotropy and hyperbolicity of quadratic forms extend to quadratic pairs; see~\cite[(6.5) and (6.12)]{Knus:1998} for the definitions. 

The following proposition summarises  the various functorial results for  hermitian forms  and involutions (resp.~quadratic forms and quadratic pairs)  we use.
 \begin{prop}\label{prop:hypiff}
Let $\psi$ and $\psi'$ be either  nondegenerate even $\lambda$-hermitian forms over $(D,\theta)$ or  quadratic forms over $F$ and let $\varphi$ be a nondegenerate symmetric bilinear form over $F$. Then 
\begin{enumerate}[$(1)$]
\item $\Ad(\psi)\simeq \Ad(\psi')$ if and only if $\psi\simeq c\psi'$ for some $c\in F^\times$.
\item
$\Ad(\psi)$ is isotropic (resp. hyperbolic) if and only if $\psi$ is isotropic (resp. hyperbolic).
\item $\Ad(\varphi\otimes \psi)\simeq \Ad(\varphi)\otimes \Ad(\psi)$.
\end{enumerate} 
\end{prop}
\begin{proof} 
$(1)$ See \cite[(4.2) and (5.11)]{Knus:1998}.

$(2)$ See  \cite[(6.3), (6.6) and (6.13)]{Knus:1998} for the result for quadratic forms. For  $\psi$  a $\lambda$-hermitian form, 
see \cite[(6.7)]{Knus:1998} for the statement on hyperbolicity. See \cite[(3.2)]{dolphin:quadpairs} for the isotropy result for  $\psi$ a nondegenerate symmetric bilinear form over $F$. The proof is easily generalised to the case of a $\lambda$-hermitian form over $(D,\theta)$. 

$(3)$ For $\psi$ a quadratic form, see  \cite[(5.19)]{Knus:1998}. Otherwise,
let $\varphi=(V,b)$ and $\psi=(W,h)$. For all $f\in \End_F(V)$, $g\in \End_D(W)$,  $v,v'\in V$ and $w,w'\in W$ we have
\begin{eqnarray*}(b\otimes h) (f\otimes g(v\otimes w), (v'\otimes w'))& 
 = &b(f(v), v')\cdot h(g(w), w') 
\\ &=&  b(v, \ad_b(f)(v'))\cdot h(w, \ad_h(g)(w'))
\\ &=& (b\otimes h) ((v\otimes w), (\ad_b(f)(v')\otimes \ad_h(g)(w')))\,.
\end{eqnarray*}
Therefore by  the bilinearity of $b\otimes h$ we have that $\ad_{b\otimes h}(f\otimes g) = \ad_b(f)\otimes \ad_h(g)$. 
Using this, it follows from the linearity of $\ad_{b\otimes h}$  that the natural isomorphism of $F$-algebras $\Phi:\End_F(V)\otimes_F\End_D(W)\rightarrow \End_{D}(V\otimes_D W)$ is an   isomorphism of the $F$-algebras with involution in the statement.
\end{proof}

It follows from \cref{prop:hypiff}$(2)$ and \cite[(1.8)]{Elman:2008} that any split algebra with symplectic involution is hyperbolic and  from \cref{prop:hypiff}$(3)$ that  an  $F$-algebra with involution adjoint to a bilinear Pfister form over $F$  is totally decomposable.


\section{Jacobson's trace forms}\label{section:jacobson}

We  now consider hermitian forms over a  quaternion division algebra or a separable quadratic extension together with their respective canonical involution. It was first shown in \cite{jacobson:hermformstr} that, over fields of characteristic different from $2$,  the theory of hermitian forms over these division algebras with involution can be reduced to the study of  quadratic forms  over the base  field that are multiples  of the respective norm forms. A version  of this result with no assumption on the characteristic of the base field was given  in  \cite{sah:evenherm}. In this case we do not get a correspondence between all hermitian forms and quadratic forms, but rather only between  even hermitian forms and quadratic forms. This is no restriction for fields of characteristic different from $2$ or for the separable quadratic extension case, as then all hermitian forms are even. 
For the cases of   a separable quadratic extension and quaternion algebras in characteristic different from two, the main result of this section is shown  in  \cite[(10.1.1) and (10.1.7)]{Scharlau:1985}. Here we give a uniform presentation for both cases and with no restriction on the characteristic of the base field.

First we recall some facts on quaternion algebras.
An $F$-\emph{quaternion algebra} is a central simple $F$-algebra of degree $2$. 
Any $F$-quaternion algebra has a basis $(1,u,v,w)$ such that
$$u^2 =u+a, v^2=b\textrm{ and }w=uv=v-vu\,$$
for some  $a\in F$ with $-4a\neq 1$, $b\in F^\times$
 (see  \cite[Chapter IX, Thm.~26]{Albert:1968}). 
  Let $Q$ be an $F$-quaternion algebra.
By  \cite[(2.21)]{Knus:1998}, the map $Q\rightarrow Q$ given by  $x\mapsto \Trd_Q(x)-x$ is the unique symplectic involution on $Q$; it is called the \emph{canonical involution of $Q$}. With an $F$-basis $(1,u,v,w)$ of $Q$ as above, the canonical involution $\gamma$ on $Q$ is determined by the conditions  $$\gamma(u)=1-u\quad\textrm{ and }\quad\gamma(v)=-v\,.$$  
  By considering $Q$ as an $F$-vector space, we can view $(Q,\mathrm{Nrd}_Q)$ as a $4$-dimensional quadratic  form on $F$. Further, $(Q,\mathrm{Nrd}_Q)$ is a $2$-fold Pfister form and $Q$ is split if and only if $(Q,\Nrd_Q)$ is hyperbolic (see \cite[(12.5)]{Elman:2008}).

  Now let   $K/F$ be a quadratic \'etale extension. That is, $K/F$ is either a separable quadratic 
extension or $K\simeq F\times F$.     
We call the non-trivial $F$-automorphism $\tau$ on $K$ the \emph{canonical involution of $K$}. 
For every quadratic \'etale extension $K/F$ there exists an element $u\in K$ satisfying 
$$\tau(u)=1-u\quad  \textrm{ and } \quad u^2-u=a$$  for some $a\in F$ with $4a\neq -1$ such that   $K$ is $F$-isomorphic to $F(u)$
(see  \cite[Chapter IX, Lemma 8]{Albert:1968} if $K$ is a field, otherwise take $u=(0,1)$). 
 Viewing $K$ as a $2$-dimensional vector-space over $F$, we may consider the $(K,\mathrm{Nrd}_K)$ as a $2$-dimensional quadratic form over $F$. 
 One can then directly check that $\Nrd_{F(u)}$ is given by $(x,y)\mapsto x^2+xy+ay^2$. In particular, $(K,\Nrd_K)$ is a $1$-fold Pfister form. Further $K\simeq F\times F$ if and only if $(K,\Nrd_K)$ is hyperbolic.




Throughout the rest of the section, let $(D,\theta)$ be either an $F$-quaternion division algebra or a separable quadratic extension of $F$ together with the respective canonical involution. 
Let $(V,h)$ be a nondegenerate even hermitian form over $(D,\theta)$.  As  $(V,h)$ is even we have that for all $x\in V$, $h(x,x)\in \Symd(D,\theta)$. Since  $\Symd(D,\theta)=F$, we may define  a  map $V\rightarrow F$ by $x\mapsto h(x,x)$. We denote this map by  $q_h$. 
 The first statement in the following result can be found in \cite[Thm.~1]{sah:evenherm}, but we provide a full proof for completeness.

\begin{prop}\label{prop:assocquadform}
Let   $(V,h)$ be   a nondegenerate even hermitian form over $(D,\theta)$.
  Considering $V$ as a vector space over $F$, the pair $(V,q_h)$ is a nonsingular quadratic form over $F$.  Further there exists a nondegenerate symmetric bilinear form $\varphi$  such that $(V,h)\simeq \varphi\otimes \qf{1}_{(D,\theta)}$ and  $(V,q_h)\simeq \varphi\otimes (D,\Nrd_D).$
\end{prop}
\begin{proof}
We clearly have that $q_h:V\rightarrow F$ is such that  $q_h(\lambda x)=\lambda^2q_h(x)$ for all $\lambda\in F$ and $x\in V$. Further, let $b:V\times V\rightarrow F$ be given by $$b(x,y)=q_h(x+y) - q_h(x)-q_h(y)=h(x,y)+ \theta(h(x,y)), \quad\textrm{for  } x,y\in V.$$ It is easily checked that $(V,b)$ is a symmetric bilinear form over $F$. Hence $(V,q_h)$ is a quadratic form over $F$. 

 By  \cite[Chapter I, (6.2.4)]{Knus:1991} there exists an orthogonal basis $(v_1,\ldots, v_n)$ of $(V,h)$ with $h(v_i,v_i)\in \Symd(D,\theta)\setminus\{0\} =F^\times$ for all $i\in\{1,\ldots, n\}$.  Let $h(v_i,v_i)=a_i$ for $i=1,\ldots, n$.
 Consider the $F$-vector space $U=Fv_1\oplus\ldots\oplus Fv_n$. Then $\varphi=(U,h|_{U\times U})$ is a nondegenerate symmetric bilinear form over $F$, and the natural isomorphism of $D$-spaces $U\otimes_FD\rightarrow V$ gives an isometry  $(V,h)\simeq\qf{a_1,\ldots, a_n}\otimes \qf{1}_{(D,\theta)}$.
 Hence for $x=(x_1,\ldots,x_n)\in V$ we have 
 $$h(x,x)= a_1\theta(x)x+\ldots +a_n\theta(x)x\,.$$
  As $\theta(x)x=\Nrd_D(x)$ for all $x\in D$  we therefore have that $(V,q_h)\simeq \varphi\otimes (D,\Nrd_D)$ for $\varphi=\qf{a_1,\ldots, a_n}$, and in particular $(V,q_h)$ is nonsingular.
\end{proof}

\begin{prop}\label{lemma:trace} 
Let $(V,h)$ be a nondegenerate even hermitian form over $(D,\theta)$.
 Then
 $(V,h)$ is isotropic (resp. hyperbolic) if and only if $(V,q_h)$ is isotropic (resp. hyperbolic). 
\end{prop}
\begin{proof}  Clearly  $(V,q_h)$ is isotropic if and only if $(V,h)$ is isotropic.
By  \cref{prop:assocquadform} there exists a nondegenerate symmetric bilinear form $\varphi$ such that
$(V,h)\simeq \varphi \otimes \qf{1}_{(D,\theta)}$. Now suppose $(V,h)$ is hyperbolic. Then $n=\dim_D(V)$ must be even and by \cref{prop:hypuniq} we may assume that $\varphi\simeq \qf{1,-1,\ldots, 1,-1}$. Then as  $(V,q_h)\simeq \varphi\otimes (D,\Nrd_D)$ 
 by  \cref{prop:assocquadform} we have that $(V,q_h)$ is hyperbolic.

Conversely, suppose $(V,q_h)$ is hyperbolic. 
Then there exists a vector $x\in V\setminus\{0\}$ such that $h(x,x)=q(x)=0$. As $(V,h)$ is nondegenerate and even, by   \cite[Chapter 1, (3.7.4)]{Knus:1991}  there exists a vector $y\in V\setminus \{0\}$ such that $h(y,y)=0$ and $h(x,y)= 1$. Let $U$ be the subspace of $V$ generated by $x$ and $y$. Then $(U,h|_U)$ is the hyperbolic $2$-dimensional hermitian form over $(D,\theta)$ and hence there exists a hermitian form $(W,h'')$ such that $(V,h)\simeq(W,h'')\perp(U,h|_U)$ by  \cite[Chapter 1, (3.7.1)]{Knus:1991}. 
 It follows that $(V,q_h)\simeq (W,q_{h''})\perp (U,q_{h}|_U)$. As $(U,q_{h}|_U)$ is hyperbolic by the first part of the proof,  it follows by  Witt cancellation, \cite[(8.4)]{Elman:2008}, that $(W,q_{h''})$ is hyperbolic, and the result follows by induction on the dimension of $V$. 
\end{proof}

\begin{cor}\label{cor:isom} Let  $(V,h)$ and $(W,h')$ be nondegenerate even hermitian forms over $(D,\theta)$.
Then  $(V,h)\simeq(W,h')$ if and only if $(V,q_h)\simeq (W,q_{h'})$.
\end{cor}
\begin{proof}
By \cref{cor:hypiso} we have  $(V,h)\simeq (W,h')$ if and only if $(V,h)\perp(-(W,h'))$ is hyperbolic. Similarly, using  \cite[(8.4)]{Elman:2008} we see  that $(V,q_h)\simeq (W,q_{h'})$ if and only if $(V,q_h)\perp(-(W,q_{h'}))$ is hyperbolic. The result  follows from \cref{lemma:trace}.
\end{proof}


\section{Trace forms  for involutions}\label{jacanal}

By \cite[(4.2)]{Knus:1998} even hermitian forms over a division algebra with symplectic (resp.~unitary) involution  correspond to algebras with symplectic (resp.~unitary) involutions. In this section we use this correspondence to  translate the hermitian form results from \S\ref{section:jacobson} to statements on symplectic and unitary involutions.

Throughout this section let $(B,\tau)$ be either an $F$-quaternion algebra or  a quadratic \'etale extension of $F$ with respective  canonical involution and let $(A,\sigma)$ be an $F$-algebra with symplectic  involution such that $A$ is Brauer equivalent to $Q$ or a split $F$-algebra with unitary involution with $Z(A)=Z(B)$ and $\s|_{Z(A)}=\tau|_{Z(B)}$ respectively.  Further, let $\pi=(B,\Nrd_B)$. 

\begin{prop}\label{prop:quadextnfullinv} 
There exists a nondegenerate  symmetric  bilinear form $\varphi$ over $F$ such that $(A,\sigma)\simeq \Ad(\varphi)\otimes  (B,\tau)$.
\end{prop}
\begin{proof} If $B$ is a split $F$-quaternion algebra or $B\simeq F\times F$ then $(A,\s)$ and $(B,\tau)$ are  hyperbolic and the result follows from \cref{cor:allhypsame}, taking $\varphi$ to be any symmetric bilinear form of the appropriate dimension.
Otherwise by \cite[(4.2)]{Knus:1998} there exists a nondegenerate  even hermitian form $\psi$ over $(B,\tau)$ such that $\Ad(\psi)\simeq (A,\s)$. By 
\cref{prop:assocquadform} there exists a  nondegenerate symmetric  bilinear form $\varphi$ over $F$ such that $\psi\simeq \varphi\otimes \qf{1}_{(B,\tau)}$. The result then follows from \cref{prop:hypiff}$(3)$.
\end{proof}

\begin{prop}\label{lemma:canon}  Let $\varphi$ be a nondegenerate symmetric bilinear form over $F$. Then 
   $\Ad(\varphi)\otimes (B,\tau)$ is isotropic (resp. hyperbolic) if and only if $\Ad(\varphi)\otimes \Ad{(\pi)}$ is isotropic (resp. hyperbolic).
\end{prop}
\begin{proof} 
If $B$ is a split $F$-quaternion algebra or $B\simeq F\times F$, then $\Ad(\varphi)\otimes (B,\tau)$ and $\pi$  are hyperbolic. Hence as is $\varphi\otimes \pi$  and  therefore also $\Ad(\varphi)\otimes \Ad {(\pi)}$ by \cref{prop:hypiff}$(2)$. Therefore we may assume that $B$ is a division $F$-algebra. 
By \cref{prop:hypiff}$(3)$ we have   
   $\Ad(\varphi\otimes \qf{1}_{(B,\tau)})\simeq\Ad(\varphi)\otimes(B,\tau)$. Since $\varphi\otimes \qf{1}_{(B,\tau)}$ is isotropic (resp.~hyperbolic) if and only if $\varphi \otimes \pi$ is isotropic (resp.~hyperbolic) by \cref{lemma:trace},  by \cref{prop:hypiff}$(2)$ we have that    $\Ad(\varphi)\otimes (B,\tau)$ is isotropic (resp.~hyperbolic) if and only if $\varphi \otimes \pi$ is isotropic (resp.~hyperbolic). This is equivalent to the isotropy (resp.~hyperbolicity) of $\Ad(\varphi\otimes\pi)$ by \cref{prop:hypiff}$(2)$.
 Finally, $\Ad(\varphi\otimes\pi)\simeq \Ad(\varphi)\otimes \Ad(\pi)$ by \cref{prop:hypiff}$(3)$, giving the result. 
\end{proof}

\begin{prop}\label{cor:isomsepqat} Let  $\varphi_1$ and $\varphi_2$  be nondegenerate symmetric bilinear forms over $F$.
Then $\Ad(\varphi_1)\otimes (B,\tau)\simeq\Ad(\varphi_2)\otimes (B,\tau)$ if and only if $\Ad(\varphi_1)\otimes\Ad(\pi) \simeq\Ad(\varphi_2)\otimes \Ad(\pi)$.
\end{prop}
\begin{proof}
By Proposition \ref{prop:hypiff}$(1)$ and $(3)$
we have  $\Ad(\varphi_1)\otimes\Ad(\pi) \simeq\Ad(\varphi_2)\otimes \Ad(\pi) $ if and only if 
$\varphi_1\otimes \pi\simeq c\varphi_2\otimes \pi$ for some $c\in F^\times$. This is equivalent to 
$\varphi_1\otimes \qf{1}_{(D,\theta)}\simeq c \varphi_2\otimes  \qf{1}_{(D,\theta)}$
by \cref{cor:isom} which is further equivalent to  $\Ad(\varphi_1)\otimes (B,\tau)\simeq\Ad(\varphi_2)\otimes (B,\tau)$ by 
\cref{prop:hypiff}$(1)$ and $(3)$.
\end{proof}

In the case where $(B,\tau)$ is an $F$-algebra with symplectic involution, we get the following reformation of Propositions \ref{lemma:canon} and \ref{cor:isomsepqat}.

\begin{cor}\label{lemma:canonnicer}  Assume  $(B,\tau)$ is an $F$-algebra with symplectic involution 
Then $(A,\sigma)$ is isotropic (resp. hyperbolic) if and only if ${(A,\sigma)}\boxtimes {(B,\tau)}$ is isotropic (resp. hyperbolic).  Further,
 let $(A',\s')$ be an $F$-algebra with symplectic involution such that $A'$ is Brauer equivalent to $B$. 
 Then $(A,\s)\simeq (A',\s')$ if and only if ${(A,\sigma)}\boxtimes {(B,\tau)}\simeq {(A',\sigma')}\boxtimes {(B,\tau)}$.
\end{cor}
\begin{proof}  
By \cite[(2.9)]{dolphin:conic} we have  $\Ad(\pi)\simeq (B,\tau)\boxtimes {(B,\tau)}$. As $(A,\s)\simeq \Ad(\varphi)\otimes (B,\tau)$ for some nondegenerate  symmetric  bilinear form $\varphi$ over $F$  by \cref{prop:quadextnfullinv}, the 
 result follows immediately from Propositions \ref{lemma:canon} and \ref{cor:isomsepqat}.
 \end{proof}

\section{Totally Decomposable Involutions}\label{decompinv}

We now give our results on totally decomposable symplectic and unitary involutions. We first recall the notion of a function field of a quadratic form.
Let $\rho$ be a nonsingular  quadratic form over $F$.
If $\dim(\rho)\geqslant 3$ or if $\rho$ is anisotropic and $\dim(\rho)=2$, then we call the function field of the projective quadric over $F$ given by $\rho$ the \emph{function field of $\rho$} and denote it by $F(\rho)$. In the remaining cases we set $F(\rho)=F$.
This agrees with the definition in \cite[\S22]{Elman:2008}.




\begin{thm}\label{prop:symp} 
Let $(Q,\gamma)$ be an $F$-quaternion algebra with canonical involution and 
 $(A,\sigma)$ be an $F$-algebra with symplectic involution such that $A$ is Brauer equivalent to $Q$ and $\deg(A)=2^n$  with $n\geqslant 1$.  The following are equivalent:
\begin{itemize}
\item[$(i)$] $(A,\sigma)$ is totally  decomposable.
\item[$(ii)$] ${(A,\sigma)}\boxtimes {(Q,\gamma)}$  is adjoint to a Pfister form.
\item[$(iii)$] $(A,\sigma)\simeq \Ad(\psi)\otimes (Q,\gamma)$ for some  bilinear Pfister form $\psi$ over $F$.
\item[$(iv)$] For any field extension $K/F$, $(A,\sigma)_K$ is either anisotropic or hyperbolic.
\end{itemize}
\end{thm}
\begin{proof} 
$(i) \Rightarrow (ii)$:  If $\kar(F)\neq2$ then  $(A,\sigma)\boxtimes (Q,\gamma)$ is equivalent to a totally decomposable orthogonal involution, and the result is a special case of  \cite[Thm.~1]{Becher:qfconj}.  If $\kar(F)=2$ then the result is a special case of \cite[(7.3)]{dolphin:PFC}.

$(ii)\Rightarrow (iii)$: Assume that ${(A,\sigma)}\boxtimes{(Q,\gamma)}\simeq \Ad(\rho)$ for a Pfister form $\rho$ over $F$. 
As $Q_{F(\pi)}$ is split, we have that $(A,\s)_{F(\pi)}$ is hyperbolic and hence   $((A,\s)\boxtimes(Q,\gamma))_{F(\pi)}$ is hyperbolic by   \cref{lemma:canonnicer}. Hence $\rho_{F(\pi)}$ is hyperbolic by \cref{prop:hypiff}$(2)$ and  by  \cite[(23.6)]{Elman:2008}  we  have  $\rho\simeq \psi\otimes\pi$ for some  bilinear Pfister form $\psi$ over $F$. As  $\Ad(\pi)\simeq (Q,\gamma)\boxtimes {(Q,\gamma)}$ by  \cite[(2.9)]{dolphin:conic} we have   ${(A,\sigma)}\boxtimes{(Q,\gamma)}\simeq\Ad(\psi)\otimes (Q,\gamma)\boxtimes (Q,\gamma)$ by \cref{prop:hypiff}$(3)$. Therefore $(A,\s)\simeq \Ad(\psi)\otimes (Q,\gamma)$ by 
\cref{lemma:canonnicer}.

  

$(iii)\Rightarrow (i)$: This follows easily from \cref{prop:hypiff}$(3)$.

$(ii)\Rightarrow (iv)$: Let $K/F$ be a field extension.
As $((A,\s)\boxtimes (Q,\gamma))_K$ is adjoint to a Pfister form, it is either anisotropic or hyperbolic by Propositions
\ref{Pfister}  and \ref{prop:hypiff}$(2)$. Hence $(A,\s)_K$ is either anisotropic or hyperbolic by \cref{lemma:canonnicer}.


$(iv)\Rightarrow (ii)$: For a given field extension $K/F$, $(A,\sigma)_K$ is anisotropic or hyperbolic if and only if the same holds for $({(A,\sigma)}\boxtimes {(Q,\gamma)})_K$ by \cref{lemma:canonnicer}. 
Let $\rho$ be a quadratic form over $F$ such that ${(A,\sigma)}\boxtimes{(Q,\gamma)}\simeq \Ad(\rho)$.
Then $\dim(\rho)=\deg(A)=2^n$ and 
we have that $\rho_K$ is either anisotropic or hyperbolic by \cref{prop:hypiff}$(2)$. As this holds for all field extensions $K/F$, we have that $\rho$ is similar to  a Pfister form  by \cref{Pfister} and the result follows from \cref{prop:hypiff}$(1)$.
\end{proof}

\begin{thm} Let $(A,\sigma)$ be a split $F$-algebra with unitary involution with centre $K$ and $\deg(A)=2^n$  with $n\geqslant 1$ and let $\tau$ be the non-trivial $F$-automorphism on $K$. The following are equivalent:
\begin{itemize}
\item[$(i)$] $(A,\sigma)$ is totally decomposable.
\item[$(ii)$] $(A,\sigma)\simeq\Ad(\psi)\otimes (K,\tau)$ for some  bilinear Pfister form $\psi$ over $F$.
\item[$(iii)$] For any field extension $L/F$, $(A,\sigma)_L$ is either anisotropic or hyperbolic.
\end{itemize}
\end{thm}
\begin{proof}
If $K$ is not a field then $(A,\s)$ and $(K,\tau)$ are both hyperbolic and by \cref{cor:allhypsame} we have $(A,\s)\simeq \Ad(\psi)\otimes (K,\tau)$ for any symmetric bilinear form $\psi$ with $\dim(\psi)=\deg(A)$. Hence the equivalencies all hold trivially. We therefore assume  that $K$ is a field and let $\pi=(K,\Nrd_K)$.

 $(i)\Rightarrow (ii)$: By  \cite[(2.22)]{Knus:1998} 
there exist $F$-quaternion algebras with involution of the first kind $(Q_i,\sigma_i)$ for $i=1,\ldots, n$ such that 
 $$(A,\sigma)\simeq (Q_1,\sigma_1)\otimes\ldots\otimes(Q_n,\sigma_n)\otimes(K,\tau)\,.$$ Again by \cite[(2.22)]{Knus:1998}, we may  assume that $(Q_i,\sigma_i)$ is symplectic for  $i=1,\ldots,n$, and hence we may assume that  $(B,\sigma')=(Q_1,\sigma_1)\otimes\ldots\otimes(Q_n,\sigma_n)$ is symplectic if $\kar(F)=2$ by  \cite[(2.23)]{Knus:1998}.
Since $B_K$ is split by the hypothesis, it follows from \cite[(98.5)]{Elman:2008} that $B$ is Brauer equivalent to an $F$-quaternion algebra $Q$. It follows by  \cite[Thm.~2]{Becher:qfconj} if $(B,\sigma')$ is orthogonal and by  \cref{prop:symp} if $(B,\sigma')$ is symplectic that $(B,\sigma')\simeq  \Ad(\varphi)\otimes (Q,\delta)$ for some bilinear Pfister form $\varphi$ over $F$, and some $F$-quaternion algebra with involution $(Q,\delta)$.
Hence $(A,\sigma)\simeq  \Ad(\varphi)\otimes (Q,\delta)\otimes (K,\tau)$. We have that $ (Q,\delta)\otimes (K,\tau)\simeq  \Ad(\psi)\otimes (K,\tau)$ for some $2$-dimensional symmetric bilinear form $\psi$ over $F$ by  \cite[(6.2)]{dolphin:quadpairs}. Hence by \cref{prop:hypiff}$(3)$, $(A,\sigma)\simeq  \Ad(\varphi\otimes\psi)\otimes (K,\tau)$ for the   bilinear Pfister form $\varphi\otimes\psi$ over $F$.

$(ii)\Rightarrow (i)$: This follows  from \cref{prop:hypiff}$(3)$.

$(ii)\Rightarrow (iii)$: Let $L/F$ be a field extension. 
If $Z(A_L)$ is not a field then $(A,\s)_L$ is hyperbolic. Otherwise  $(A,\s)_L$ is  anisotropic or hyperbolic if and only if $(\psi\otimes \pi)_L$ is hyperbolic
by Propositions  \ref{lemma:canon} and \ref{prop:hypiff}$(2)$. The result follows from  \cref{Pfister}.

$(iii)\Rightarrow (ii)$: By  \cref{prop:quadextnfullinv}  we have $(A,\sigma)\simeq \Ad(\varphi)\otimes (K,\tau)$ for some symmetric bilinear form $\varphi$ over $F$ with $\dim(\varphi)=\deg(A)$. Let $\rho=\varphi\otimes \pi$. 
Let $L/K$ be a field extension. 
If $K\otimes_FL=Z(A_L)$ is not a field then $\pi_L$ is hyperbolic.
Otherwise $\rho_L$ is anisotropic or hyperbolic if and only if $(A,\s)_L$ is anisotropic or hyperbolic  by \cref{lemma:canon}. In both cases, $\rho_L$ is  either anisotropic or hyperbolic.  As this holds for any field extension $L/F$,  $\rho$ is similar to a Pfister form by \cref{Pfister}. Similarly,  $\rho_{F(\pi)}$ is isotropic, and hence hyperbolic. It then follows from \cite[(23.6) and (24.1)]{Elman:2008} that $\rho\simeq c\psi\otimes\pi$
 for some bilinear Pfister form $\psi$ over $F$ and some $c\in F^\times$. 
 Hence  by \cref{prop:hypiff}$(1)$, $\Ad(\rho)\simeq \Ad(\varphi)\otimes \Ad(\pi)\simeq \Ad(\psi)\otimes \Ad(\pi)$. 
It follows that  $\Ad(\varphi)\otimes (K,\tau)\simeq \Ad(\psi)\otimes (K,\tau)$ by  \cref{cor:isomsepqat}, as required.\end{proof}

\small{}

\end{document}